\documentclass{article}
\usepackage{amsmath}
\usepackage{amssymb}
\usepackage{amsthm}
\usepackage{algorithm}
\usepackage{algorithmic}
\usepackage{multirow}% added 2019/10/31
\usepackage{graphicx}
\usepackage{color}
\usepackage{lscape}
\def\@themcountersep{}
\newtheorem{theorem}{Theorem}[section]

\newtheorem{corollary}[theorem]{Corollary}
\newtheorem{definition}[theorem]{Definition}

\newtheorem{lemma}[theorem]{Lemma}
\newtheorem{proposition}[theorem]{Proposition}
\newtheorem{remark}[theorem]{Remark}
\def\ba{\begin{array}}
\def\ea{\end{array}}
\def\beq{\begin{equation}}
\def\eeq{\end{equation}}
\def\beann{\begin{eqnarray*} }
\def\eeann{\end{eqnarray*}}
\def\bc{\begin{center}}
\def\ec{\end{center}}
\def\bea{\begin{eqnarray}}
\def\eea{\end{eqnarray}}

\begin{document}

%%%%%%%%%%%%%%%%%%%%%%%%%%%%%%%%%%%%%%%%%%%%%%%%%%%%%%%%%%% Title
%% \input title

%%%%% 
\title{Polyhedral approximations of the semidefinite cone and their application\thanks{
An earlier version of this paper was entitled ``Polyhedral approximations of the semidefinite cone and their applications.'' 
This research was supported by the Japan Society for the Promotion of Science through a Grant-in-Aid for Challenging Exploratory Research (17K18946) and a Grant-in-Aid for Scientific Research ((B)19H02373) of the Ministry of Education, Culture, Sports, Science and Technology of Japan.}
}
%%%%%

\author{Yuzhu Wang\thanks{
Graduate School of Systems and Information Engineering, University of Tsukuba, Tsukuba, Ibaraki 305-8573, Japan. email: s1930138@s.tsukuba.ac.jp 
}, 
Akihiro Tanaka\thanks{
Central Research Institute of Electric Power Industry, Yokosuka, Kanagawa 240-0196, Japan. 
%Graduate School of Systems and Information Engineering, University of Tsukuba, Tsukuba, Ibaraki 305-8573, Japan. 
email: a-tanaka@criepi.denken.or.jp
}   
 and 
Akiko Yoshise\thanks{Corresponding author. Faculty of Engineering, Information and Systems, University of Tsukuba, Tsukuba, Ibaraki 305-8573, Japan. email: yoshise@sk.tsukuba.ac.jp
}     
}

\date{May 2019 \\ Revised December 2020}

\maketitle

\begin{abstract}
We develop techniques to construct a series of sparse polyhedral approximations of the semidefinite cone. Motivated by the semidefinite (SD) bases proposed by Tanaka and Yoshise (2018), we propose a simple expansion of SD bases so as to keep the sparsity of the matrices composing it. We prove that the polyhedral approximation using our expanded SD bases contains the set of all diagonally dominant matrices and is contained in the set of all scaled diagonally dominant matrices. We also prove that the set of all scaled diagonally dominant matrices can be expressed using an infinite number of expanded SD bases. We use our approximations as the initial approximation in cutting plane methods for solving a semidefinite relaxation of the maximum stable set problem. It is found that the proposed methods with expanded SD bases are significantly more efficient than methods using other existing approximations or solving semidefinite relaxation problems directly. 
\end{abstract}

% REQUIRED

% REQUIRED

\section{Introduction}

A semidefinite optimization problem (SDP) is an optimization problem in variables in the space of symmetric matrices with a linear objective function and linear constraints over the semidefinite cone. We denote the space of symmetric matrices as ${\mathbb S}^n:=\{X\in\mathbb{R}^{n\times n}\mid X_{i,j}=X_{j,i} \ (1 \leq i < j \leq n) \}$ and the semidefinite cone as ${\cal S}^n_+:=\{X\in\mathbb{S}^n\mid d^TXd\ge0 \ \mbox{for any} \ d\in\mathbb{R}^n \}$. Accordingly, we can readily define an SDP in the standard form, as 
\begin{align}\label{SDP}
\min&\ \langle C,X\rangle\nonumber\\
{\rm s.t.}&\ \langle A_j,X\rangle=b_j,j=1,2,\ldots,m,\\
&\ X\in{\cal S}^n_+,\nonumber
\end{align}
where $C\in\mathbb{S}^n$, $A_j\in\mathbb{S}^n$, $b_j\in\mathbb{R}$ ($j=1,2,\ldots,m$), and $\langle A,B\rangle:={\rm Trace}(A^TB)=\sum_{i,j=1}^nA_{i,j}B_{i,j}$ is the inner product over $\mathbb{S}^n$. 

SDPs are powerful tools that provide convex relaxations for combinatorial and nonconvex optimizations, such as the max-cut problem (e.g., \cite{Goemans:1995:IAA:227683.227684}, \cite{doi:10.1080/10556780108805818}) and the k-equipartition problem (e.g., \cite{wolkowicz2012handbook}, \cite{karisch1998semidefinite}). Some of these relaxations can even attain the optimum, as shown in \cite{Lasserre2001} and \cite{Kim2003}. Interested readers may find details about SDPs and their relaxations in \cite{wolkowicz2012handbook}, \cite{todd2001} and \cite{laurent2012semidefinite}. 

A cone ${\cal K} \subset \mathbb{S}^n$ is called proper if it has a non-empty interior and is closed, pointed (i.e., ${\cal K} \cap -{\cal K}=\{O\}$), and convex. It is known that the SDP cone is a proper cone \cite{blekherman2012semidefinite}. By replacing the semidefinite constraint $X\in{\cal S}^n_+$ with a general conic constraint $X\in{\cal K}$ in (\ref{SDP}) (say, a proper cone ${\cal K} \subset \mathbb{S}^n$), one can obtain a general class of problems, namely, conic optimization problems. The class of conic optimization problems has been an active field of study because it contains many popular classes of problems, including linear optimization problems (LPs), second-order cone programs (SOCPs), SDPs, and copositive programs. Copositive programs have been shown capable of providing tight lower bounds for combinatorial and quadratic optimization problems, as described in the survey paper by D{\"u}r \cite{dur2010copositive} and the recent work of Arima et al. \cite{arima2013quadratically}, \cite{kim2016lagrangian}, \cite{arima2017robust}, etc. It has been shown that a copositive relaxation sometimes gives a highly accurate approximate solution for some combinatorial problems under certain conditions \cite{arima2018lagrangian}, \cite{burer2009copositive}. However, the copositive program and its dual problem are both NP-hard (see, e.g., \cite{dickinson2014computational} and \cite{murty1987some}). 

SDPs are also attractive because they can be solved in polynomial time to any desired precision. There are state-of-the-art solvers, such as SDPA \cite{SDPA2003}, SeDuMi \cite{sturm1999using}, SDPT3  \cite{toh1999sdpt3}, and Mosek \cite{mosek}, but their computations become difficult when the size of the SDP becomes large. To overcome this deficiency, for example, one may use preprocessing to reduce the size of the SDPs, which leads to facial reduction methods \cite{permenter2014partial}, \cite{7040427} and \cite{Waki2013}. As another idea, one may generate relaxations of SDPs and solve them as easily handled optimization problems, e.g., LPs and SOCPs, which leads to cutting plane methods. We will focus on these latter methods. 

The cutting plane method solves an SDP by transforming it into an optimization problem (e.g., an LP or an SOCP), adding cutting planes at each iteration to cut the current approximate solution out of the feasible region in the next iterations and to get close to the optimal value. The cutting plane method was first used on the traveling-salesman problem, by Dantzig, Fulkerson, and Johnson \cite{dantzig1954solution}, \cite{dantzig1959linear} in 1954. It was used in 1958 by Gomory \cite{gomory1958outline} to solve integer linear programming problems. As SDPs became popular, it came to be used on them as well; see, for instance, Krishnan and Mitchell \cite{krishnan2002linear}, \cite{krishnan2006unifying} and \cite{Krishnan2006}, and Konno et al. \cite{konno2002cutting}. Kobayashi and Takano \cite{kobayashibranch} applied it to a class of mixed-integer SDPs. In \cite{ahmadi2017optimization}, Ahmadi, Dash, and Hall applied it to nonconvex polynomial optimization problems and copositive programs. 

In the above-mentioned cutting plane methods for SDPs, the semidefinite constraint $X\in{\cal S}^n_+$ in (\ref{SDP}) is first relaxed to $X\in{\cal K}_{\rm out}$, where ${\cal S}^n_+\subseteq{\cal K}_{\rm out}\subseteq{\mathbb S}^n$, and an initial relaxation of the SDP is obtained. If ${\cal K}_{\rm out}$ is polyhedral, the initial relaxation may give an LP; if ${\cal K}_{\rm out}$ is given by second-order constraints, the initial relaxation becomes an SOCP. To improve the performance of these cutting plane methods, we consider generating initial relaxations for SDPs that are both tight and computationally efficient and focus on approximations of ${\cal S}^n_+$.

Many approximations of ${\cal S}^n_+$ have been proposed on the basis of its well-known properties. Kobayashi and Takano \cite{kobayashibranch} used the fact that the diagonal elements of semidefinite matrices are nonnegative. Konno et al. \cite{konno2002cutting} imposed an assumption that all diagonal elements of the variable $X$ in the SDPs appearing in their iterative algorithm are bounded by a constant. The sets of diagonally dominant matrices and scaled diagonally dominant matrices are known to be cones contained in ${\cal S}^n_+$, (see, e.g., \cite{horn1990matrix} and \cite{ahmadi2017optimization} for details). The inclusive relation among them has been studied in, e.g., \cite{berman1994nonnegative} and \cite{bishan1998iterative}. Ahmadi et al. \cite{ahmadi2017optimization} and \cite{ahmadi2017dsos} used these sets as initial approximations of their cutting plane method. Boman et al. \cite{boman2005factor} defined the {\em factor width} of a semidefinite matrix, and Permenter and Parrilo used it to generate approximations of ${\cal S}^n_+$, which they applied to facial reduction methods in \cite{permenter2014partial}. 

Tanaka and Yoshise defined various bases of $\mathbb{S}^n$, wherein each basis consists of $\frac{n(n+1)}{2}$ semidefinite matrices, called semidefinite (SD) bases, and used them to devise approximations of ${\cal S}^n_+$ \cite{tanaka2018lp}. They showed that the conical hull of SD bases and its dual cone give inner and outer polyhedral approximations of ${\cal S}^n_+$, respectively. On the basis of the SD bases, they also developed techniques to determine whether a given matrix is in the semidefinite plus nonnegative cone ${\cal S}^n_++{\cal N}^n$, which is the Minkowski sum of ${\cal S}^n_+$ and the nonnegative matrices cone ${\cal N}^n$. In this paper, we focus on the fact that SD bases are sometimes sparse, i.e., the number of nonzero elements in a matrix is relatively small, and hence, it is not so computationally expensive to solve polyhedrally approximated problems in such SD bases. We call such an approximation, a {\em sparse polyhedral approximation}, and propose efficient sparse approximations of ${\cal S}^n_+$.

The goal of this paper is to construct tight and sparse polyhedral approximations of ${\cal S}^n_+$ by using SD bases in order to solve hard conic optimization problems, e.g., doubly nonnegative (DNN, or ${\cal S}^n_+ \cap \mathcal{N}^n$) and semidefinite plus nonnegative ($\mathcal{S}^n_+ + \mathcal{N}^n$) optimization problems. The contributions of this paper are summarized as follows.
\begin{itemize}
\item This paper gives the relation between the conical hull of sparse SD bases and the set of diagonally dominant matrices. We propose a simple expansion of SD bases without losing the sparsity of the matrices and prove that one can generate a sparse polyhedral approximation of ${\cal S}^n_+$ that contains the set of diagonally dominant matrices and is contained in the set of scaled diagonally dominant matrices. 
\item The expanded SD bases are used by cutting plane methods for a semidefinite relaxation of the maximum stable set problem. It is found that the proposed methods with expanded SD bases are significantly more efficient than methods using other approximations or solving semidefinite relaxation problems directly. 
\end{itemize}

The organization of this paper is as follows. Various approximations of ${\cal S}^n_+$ are introduced in section \ref{sec:1}, including those based on the factor width by Boman et al. \cite{boman2005factor}, diagonal dominance by Ahmadi et al. \cite{ahmadi2017optimization}, and SD bases by Tanaka and Yoshise \cite{tanaka2018lp}. The main results of this paper, i.e., an expansion of SD bases and an analysis of its theoretical properties, are provided in section \ref{sec:5}. In section \ref{sec:9}, we introduce the cutting plane method using different approximations of ${\cal S}^n_+$ for calculating upper bounds of the maximum stable set problem. We also describe the results of numerical experiments and evaluate the efficiency of the proposed method with expanded SD bases.

\section{Some approximations of the semidefinite cone}
\label{sec:1}

\subsection{Factor width approximation}
\label{sec:2}

In \cite{boman2005factor}, Boman et al. defined a concept called factor width.
\begin{definition}\label{facwid}
\emph{(Definition 1 in \cite{boman2005factor})}
The factor width of a real symmetric matrix $A\in\mathbb{S}^n$ is the smallest integer $k$ such that there exists a real matrix $V\in\mathbb{R}^{n\times m}$ where $A=VV^T$ and each column of $V$ contains at most $k$ nonzero elements.
\end{definition}

For $k\in\{1,2,\ldots,n\}$, we can also define
\begin{align*}
{\cal FW}(k):=\{X\in\mathbb{S}^n\mid \text{X has a factor width of at most }k\}.
\end{align*}

It is obvious that the factor width is only defined for semidefinite matrices, because for every matrix $A$ in Definition \ref{facwid}, the decomposition $A=VV^T$ implies that $A\in{\cal S}^n_+$. Therefore, for every $k\in\{1,2,\ldots,n\}$, the set of matrices with a factor width of at most $k$ gives an inner approximation of $\mathbb{S}^n_+$: ${\cal FW}(k)\subseteq\mathbb{S}^n_+.$

\subsection{Diagonal dominance approximation}
\label{sec:3}
In \cite{ahmadi2017optimization} and \cite{ahmadi2017dsos}, the authors approximated the cone ${\cal S}^n_+$ with the set of diagonally dominant matrices and the set of scaled diagonally dominant matrices.

\begin{definition}\label{dd}
The set of diagonally dominant matrices ${\cal DD}_n$ and the set of scaled diagonally dominant matrices ${\cal SDD}_n$ are defined as follows:
\begin{align*}
{\cal DD}_n&:=\{A\in\mathbb{S}^n\mid \ A_{i,i}\ge\sum_{j\neq i}|A_{i,j}|\hspace{3mm} (i=1,2,\ldots,n)\},\\
{\cal SDD}_n&:=\{A\in\mathbb{S}^n\mid DAD\in{\cal DD}_n \ \mbox{\em for some positive diagonal matrix $D$} \}.
\end{align*}
\end{definition}

It is easy to see that ${\cal DD}_n$ is a convex cone and ${\cal SDD}_n$ is a cone in $\mathbb{S}^n$. As a consequence of the Gershgorin circle theorem \cite{gershger}, we have the relation ${\cal DD}_n\subseteq{\cal SDD}_n\subseteq{\cal S}^n_+$. Ahmadi et al. \cite{ahmadi2017optimization} defined ${\cal U}_{n,k}$ as the set of vectors in $\mathbb{R}^n$ with at most $k$ nonzeros, each equal to $1$ or $-1$. They also defined a set of matrices $U_{n,k}:=\{uu^T\mid u\in{\cal U}_{n,k}\}$. Barker and Carlson \cite{barker1975cones} proved the following theorem.
\begin{theorem}\label{unk}
\emph{(Barker and Carlson \cite{barker1975cones})} ${\cal DD}_n={\rm cone}(U_{n,2}).$
\end{theorem}

The conical hull of a given set ${\cal K}\subseteq\mathbb{S}^n$ is defined as ${\rm cone}({\cal K}):=\{\sum_{i=1}^k \alpha_iX_i\mid X_i\in {\cal K},\alpha_i\ge0,k\in{\mathbb Z}_{\ge0}\}$, where $\mathbb{Z}_{\ge0}$ is the set of nonnegative integers. A cone generated in this way by a finite number of elements is called {\it finitely generated}. Theorem \ref{unk} implies that ${\cal DD}_n$ has $n^2$ extreme rays; thus, it is a finitely generated cone. 

A cone ${\cal K}\in\mathbb{S}^n$ is {\it polyhedral} if ${\cal K}=\{X\in\mathbb{S}^n\mid\langle A_i,X\rangle\leq 0\}$ for some $A_i\in\mathbb{S}^n$. The following theorem follows from the results of Minkowski \cite{minkowski1896} and Weyl \cite{Weyl1935}.
\begin{theorem}\label{minweyl}
\emph{(Minkowski-Weyl theorem, see Corollary 7.1a in \cite{schrijver1998theory})} A convex cone is polyhedral if and only if it is finitely generated. 
\end{theorem}

The above theorem ensures that ${\cal DD}_n$ is a polyhedral cone. Using the expression in Theorem \ref{unk}, Ahmadi et al. proved that optimization problems over ${\cal DD}_n$ can be solved as LPs. They also proved that optimization problems over ${\cal SDD}_n$ can be solved as SOCPs. They designed a column generation method using ${\cal DD}_n$ and ${\cal SDD}_n$ to obtain a series of inner approximations of ${\cal S}_n^+$. As for the relation between the factor width and diagonal dominance, useful results were presented in \cite{boman2005factor} and in \cite{ahmadi2017dsos}, which gives a relation between ${\cal SDD}_n$ and the set of matrices with a factor width of at most $2$.

\begin{lemma}\label{factorSDDn}
\emph{(See \cite{boman2005factor} and Theorem 8 in \cite{ahmadi2017dsos})} ${\cal FW}(2)={\cal SDD}_n$
\end{lemma}

Note that Definition \ref{facwid} implies that the set ${\cal FW}(k)$ is convex for any $k \in \{1,2,\ldots,n\}$, and we obtain the following corollary as Lemma \ref{factorSDDn}:

\begin{corollary}\label{convexSDDn}
The set ${\cal SDD}_n$ is a convex cone.
\end{corollary}

\subsection{SD basis approximation}
\label{sec:4}

Tanaka and Yoshise defined semidefinite (SD) bases \cite{tanaka2018lp}.
\begin{definition}\label{df:sdbasis}
\emph{(Definitions 1 and 2 in \cite{tanaka2018lp})}
Let $e_i\in\mathbb{R}^n$ denotes the vector with a $1$ at the $i$th coordinate and $0$ elsewhere, and let $I=(e_1,\ldots,e_n)\in\mathbb{S}^n$ be the identity matrix. Then
\begin{align*}
{\cal B_+}:=\{(e_i+e_j)(e_i+e_j)^T\mid 1\leq i\leq j\leq n\}
\end{align*}
is called an SD basis of Type I, and 
\begin{align*}
{\cal B_-}:=\{(e_i+e_i)(e_i+e_i)^T\mid 1\leq i\leq n\}\cup\{(e_i-e_j)(e_i-e_j)^T\mid 1\leq i<j\leq n\}
\end{align*}
is called an SD basis of Type II. Matrices in SD bases Type I and II are defined as
\begin{align*}
{B}^+_{i,j:}=(e_i+e_j)(e_i+e_j)^T,\ {B}^-_{i,j}:=(e_i-e_j)(e_i-e_j)^T.
\end{align*}
\end{definition}

As shown in \cite{tanaka2018lp}, ${\cal B}_+$ and ${\cal B}_-$ are subsets of $\mathcal{S}^n_+$ and bases of $\mathbb{S}^n$. Given a set ${\cal K}\subseteq\mathbb{S}^n$, we define the dual cone of ${\cal K}$ as $({\cal K})^*:=\{A\in\mathbb{S}^n\mid \langle A,B\rangle\ge0 \ \mbox{for any} \ B\in{\cal K}\}$. The conical hull of ${\cal B}_+ \cup {\cal B}_-$ and its dual give an inner and an outer polyhedral approximation of $\mathcal{S}^n_+$, as follows.
\begin{definition}\label{Sin}
Let $I=(e_1,\ldots,e_n)\in\mathbb{S}^n$ be the identity matrix. The inner and outer approximations of ${\cal S}^n_+$ by using SD bases are defined as 
\begin{align*}
{\cal S}_{\rm in}:={\rm cone}({\cal B}_+\cup{\cal B}_-),\ \ {\cal S}_{\rm out}:=({\cal S}_{\rm in})^*.
\end{align*}
\end{definition}

By Definition \ref{df:sdbasis}, we know that ${\cal B}_+,{\cal B}_-\subseteq{\cal S}^n_+$. Since ${\cal S}^n_+$ is a convex cone, we have ${\cal S}_{\rm in}\subseteq{\rm cone}({\cal S}^n_+)={\cal S}^n_+$. By Lemma 1.7.3 in \cite{laurent2012semidefinite}, we know that ${\cal S}^n_+$ is self-dual; that is, ${\cal S}^n_+=({\cal S}^n_+)^*$. Accordingly, we can conclude that ${\cal S}_{\rm in}\subseteq{\cal S}^n_+\subseteq{\cal S}_{\rm out}$.

\begin{remark} \label{PtoI}
In \cite{tanaka2018lp}, ${\cal B}_+$ and ${\cal B}_-$ are defined as  ${\cal B}_+(P)$ and ${\cal B}_-(P)$ using an orthogonal matrix $P$ instead of the identity matrix $I$.
In fact, for any orthogonal matrix $P$, 
\[
P{\cal B_+}P^T :=\{ PB_{i,j}^+P^T \mid B_{i,j}^+ \in {\cal B_+} \} \ \mbox{and} \ 
P{\cal B_-}P^T :=\{ PB_{i,j}^-P^T \mid B_{i,j}^- \in {\cal B_-} \} 
\]
also give other bases and generalizations of  ${\cal B_+}$ and ${\cal B_-}$.  
However, as we will see in section \ref{sec:9}, we use the matrices in the bases as in optimization problems of the form
\[
\min \ \langle C,X\rangle\nonumber 
\ \ {\rm s.t.} \ \langle A,X\rangle=b, \  \langle Y ,X\rangle\ge0 \ (Y \in {\cal B_+}),
\]
which is equivalent to
\begin{equation} \label{P}
\min \ \langle PCP^T,\bar{X}\rangle
\ \ {\rm s.t.} \ \langle PAP^T, \bar{X}\rangle=b, \  \langle Y ,\bar{X}\rangle\ge0 \ (Y \in P{\cal B_+}P^T).
\end{equation}
Therefore, we consider that the generalizations  $P{\cal B_+}P^T$ and $P{\cal B_-}P^T$ are not essential throughout this paper and omit those descriptions from subsequent sections to simplify the presentation.
\end{remark}

\section{Expansion of SD bases}
\label{sec:5}

When we use the SD bases for approximating $\mathcal{S}^n_+$, the sparsity of the matrices in those bases is quite important in terms of computational efficiency. As we mentioned in Remark \ref{PtoI}, for any orthogonal matrix $P$, $P{\cal B_+}P^T$ and $P{\cal B_-}P^T$ give generalizations of the SD bases. However, it is hard to choose an appropriate orthogonal matrix $P$ (except for the identity matrix $I$) to keep the sparsity of the matrices $PCP^T$ and $PAP^T$ in (\ref{P}). In this section, we try to extend the definition of the SD bases in order to obtain various sparse SD bases which will lead us to sparse polyhedral approximations of $\mathcal{S}^n_+$.

\subsection{SD bases and their relations with ${\cal S}^{n}_+$ and ${\cal DD}_n$}
\label{sec:6}
First, we give a lemma that provides an expression of ${\cal S}^{n}_+$ by using SD bases.
The lemma is a direct corollary of the fact that any $X \in {\cal S}^{n}_+$ has nonnegative eigenvalues and a corresponding orthogonal basis of eigenvectors.

\begin{lemma}\label{scone}
\begin{align*}
{\cal S}^n_+={\rm cone}\left( \displaystyle\bigcup_{P\in{\cal O}^n}\{P^TXP\mid X\in{\cal B}_+\}\right)={\rm cone}\left( \displaystyle\bigcup_{P\in{\cal O}^n}\{P^TXP\mid X\in{\cal B}_-\} \right),
\end{align*}
where ${\cal O}^n$ is the set of orthogonal matrices in $\mathbb{R}^{n\times n}$.
\end{lemma}

Lemma \ref{scone} gives a way to approximate ${\cal S}^n_+$ by changing the matrix $P=(p_1,..,p_n)$ $\in{\cal O}^n$ when creating SD bases. However, a dense matrix $P\in{\cal O}^n$ may lead to a dense formulation of the approximation using SD basis, which is unattractive from the standpoint of computational efficiency. 

Note that we can easily see that the set ${\rm cone}({\cal B}_+\cup{\cal B}_-)$, the conical hull of the sparse SD bases $\mathcal{B}_+$ and $\mathcal{B}_-$, is equivalent to ${\rm cone}(U_{n,2})$. Thus, we obtain the following proposition as a corollary of Theorem \ref{unk}.

\begin{proposition}\label{pro:ddsd}
\begin{align*}
{\rm cone}({\cal B}_+\cup{\cal B}_-)={\cal DD}_n.
\end{align*}
\end{proposition}

\subsection{Expansion of SD bases without losing sparsity}
\label{sec:7}

The previous section shows that we can obtain a sparse polyhedral approximation of $\mathcal{S}^n_+$ by using the SD bases. In this section, we try to extend the definition of the SD bases in order to obtain various sparse polyhedral approximations of $\mathcal{S}^n_+$.

\begin{definition}\label{df:exp}
Let $I=(e_1,\ldots,e_n)\in\mathbb{S}^n$ be the identity matrix. Define the expansion of the SD basis with one parameter $\alpha\in\mathbb{R}$ as
\begin{align*}
\bar{B}_{i,j}(\alpha)&:=( e_i+\alpha e_j)( e_i+\alpha e_j)^T,\\
\bar{\cal B}(\alpha)&:=\{\bar{B}_{i,j}(\alpha)\mid 1\leq i\leq j\leq n\}.
\end{align*}
\end{definition}

The proposition below ensures that the expansion of the SD bases also gives bases of $\mathbb{S}^n$.

\begin{proposition}\label{Independency}
Let $I=(e_1,\ldots,e_n)\in\mathbb{S}^n$ be the identity matrix. For any $\alpha\in\mathbb{R}\setminus\{0,-1\}$, $\bar{\cal B}(\alpha)$ is a set of $n(n+1)/2$ independent matrices and thus a basis of $\mathbb{S}^n$.
\end{proposition}

\begin{proof}

Let $\alpha\in\mathbb{R}\setminus\{0,-1\}$. Accordingly, for $1\leq i<j\leq n$, we have 
\begin{align}
\bar{B}_{i,j}(\alpha):=&( e_i+\alpha e_j)( e_i+\alpha e_j)^T\nonumber\\
=&e_ie_i^T+\alpha(e_ie_j^T+e_je_i^T)+\alpha^2e_je_j^T\nonumber\\
=&\alpha(e_ie_i^T+e_ie_j^T+e_je_i^T+e_je_j^T)+(1-\alpha)e_ie_i^T+(\alpha^2-\alpha)e_je_j^T\nonumber\\
=&\alpha{B}_{i,j}^++\frac{1-\alpha}{4}{B}_{i,i}^++\frac{\alpha(\alpha-1)}{4}{B}_{j,j}^+,\label{eq:bijexp}
\end{align}
and for every $1\leq i\leq n$, we also have
\begin{align}
\bar{B}_{i,i}(\alpha):=&( e_i+\alpha e_i)( e_i+\alpha e_i)^T\nonumber\\
=&(1+\alpha)^2 e_i e_i^T=\frac{(1+\alpha)^2}{4}{B}_{i,i}^+. \label{eq:bijexp2}
\end{align}

Suppose that there exist $\gamma_{i,j}\ge0\ (1\leq i\leq j\leq n)$ such that
\begin{align*}
\sum_{1\leq i\leq j\leq n}\gamma_{i,j}\bar{B}_{i,j}(\alpha)=O.
\end{align*}

Then, by (\ref{eq:bijexp}) and (\ref{eq:bijexp2}), we see that
\begin{align}
O=&\sum_{i=1}^n \frac{\gamma_{i,i}(1+\alpha)^2}{4}{B}_{i,i}^++\sum_{1\leq i<j\leq n}\gamma_{i,j}\left[\alpha{B}_{i,j}^++\frac{1-\alpha}{4}{B}_{i,i}^++\frac{\alpha(\alpha-1)}{4}{B}_{j,j}^+\right]\nonumber\\
=&\sum_{i=1}^n \frac{(1+\alpha)^2}{4}\gamma_{i,i}{B}_{i,i}^++\sum_{1\leq i<j\leq n}\alpha\gamma_{i,j}{B}_{i,j}^++\sum_{i=1}^{n-1} \frac{1-\alpha}{4}(\sum_{j=i+1}^n\gamma_{i,j}){B}_{i,i}^+\nonumber\\
&+\sum_{j=2}^n\frac{\alpha(\alpha-1)}{4}(\sum_{i=1}^{j-1}\gamma_{i,j}){B}_{j,j}^+\nonumber\\
=&\left[\frac{\gamma_{1,1}(1+\alpha)^2}{4}+\frac{1-\alpha}{4}(\sum_{j=2}^n\gamma_{1,j})\right]{B}_{1,1}^+\nonumber\\
&+\sum_{i=2}^{n-1}\left[ \frac{(1+\alpha)^2}{4}\gamma_{i,i}+\frac{1-\alpha}{4}(\sum_{j=i+1}^n\gamma_{i,j})+\frac{\alpha(\alpha-1)}{4}(\sum_{j=1}^{i-1}\gamma_{j,i})\right] {B}_{i,i}^+ \nonumber\\
&+\left[\frac{\gamma_{n,n}(1+\alpha)^2}{4}+\frac{\alpha(\alpha-1)}{4}(\sum_{j=1}^{n-1}\gamma_{j,n})\right]{B}_{n,n}^+\nonumber\\
&+\sum_{1\leq i<j\leq n}\alpha\gamma_{i,j}{B}_{i,j}^+ \label{eq:ind}.
\end{align}

Since $\{B_{i,j}^+\}={\cal B}_+$ is a set of linearly independent matrices, all the coefficients for ${B}_{i,j}$ in (\ref{eq:ind}) should be $0$. Thus, we have
\begin{align}
&0=\frac{\gamma_{1,1}(1+\alpha)^2}{4}+\frac{1-\alpha}{4}(\sum_{j=2}^n\gamma_{1,j}),\label{r1}\\
&0= \frac{(1+\alpha)^2}{4}\gamma_{i,i}+\frac{1-\alpha}{4}(\sum_{j=i+1}^n\gamma_{i,j})+\frac{\alpha(\alpha-1)}{4}(\sum_{j=1}^{i-1}\gamma_{j,i})\ (2\leq i\leq n-1),\label{r2}\\
&0=\frac{\gamma_{n,n}(1+\alpha)^2}{4}+\frac{\alpha(\alpha-1)}{4}(\sum_{j=1}^{n-1}\gamma_{j,n}),\label{r3}\\
&0=\alpha\gamma_{i,j} \ (1\leq i<j\leq n)\label{rij}.
\end{align}

Since $\alpha\neq0$, by (\ref{rij}) we have 
\begin{align}
\gamma_{i,j}=0 \ (1\leq i<j\leq n).\label{r4}
\end{align}

Since $\alpha\neq-1$, (\ref{r1})-(\ref{r4}) imply that 
\begin{align}
\gamma_{i,i}=0 \ (i=1,2,\ldots,n).\label{r5}
\end{align}

The above leads us to conclude that $\{\bar B_{i,j}(\alpha)\}=\bar{\cal B}(\alpha)$ is a set of $n(n+1)/2$ linearly independent matrices.\ $\square$
\end{proof}

If we let $\alpha=1$, then it is straightforward that $\bar{\cal B}(1)={\cal B}_+$. If we let $\alpha$ be other real numbers, we may obtain different SD bases. The following proposition gives the condition for generating different expanded SD bases. 

\begin{proposition}\label{pro:3.3.3}
Let $I=(e_1,\ldots,e_n)\in\mathbb{S}^n$ be the identity matrix. Suppose that $\alpha_1\in\mathbb{R}\setminus\{0,-1\}$ and $\alpha_2\in\mathbb{R}\setminus\{0,\alpha_1\}$. Then, for every $1\leq i<j\leq n$,
\begin{align*}
( e_i+\alpha_2 e_j)( e_i+\alpha_2 e_j)^T\notin{\rm cone}(\bar{\cal B}(\alpha_1)).
\end{align*}
\end{proposition}
\begin{proof}

For $1\leq i\leq j\leq n$, let us define
\begin{align*}
\bar{B}_{i,j}^1:=( e_i+\alpha_1 e_j)( e_i+\alpha_1 e_j)^T,\ \ \bar{B}_{i,j}^2:=( e_i+\alpha_2 e_j)( e_i+\alpha_2 e_j)^T.
\end{align*}

Note that if $i=j$, then 
\begin{align}\label{eq:Bii}
\bar{B}_{i,i}^1:=(1+\alpha_1)^2e_ie_i^T,\ \ \bar{B}_{i,i}^2:=(1+\alpha_2)^2e_ie_i^T.
\end{align}

For every $i< j$, we can write $\bar{B}_{i,j}^2$ as a linear combination of $\bar{B}_{i,j}^1$:
\begin{align}
\bar{B}_{i,j}^2=&e_ie_i^T+\alpha_2^2e_je_j^T+\alpha_2(e_ie_j^T+e_je_i^T)\nonumber\\
=&e_ie_i^T+\alpha_2^2e_je_j^T+\frac{\alpha_2}{\alpha_1}\alpha_1(e_ie_j^T+e_je_i^T)\ \ \ ({\rm because\ }\alpha_1\neq0)\nonumber\\
=&e_ie_i^T+\alpha_2^2e_je_j^T-\frac{\alpha_2}{\alpha_1}e_ie_i^T-\frac{\alpha_2\alpha_1^2}{\alpha_1}e_je_j^T\nonumber\\
&+\frac{\alpha_2}{\alpha_1}\left[e_ie_i^T+\alpha_1(e_ie_j^T+e_je_i^T)+\alpha_1^2e_je_j^T\right]\nonumber\\
=&\frac{\alpha_1-\alpha_2}{\alpha_1}e_ie_i^T+\alpha_2(\alpha_2-\alpha_1)e_je_j^T+\frac{\alpha_2}{\alpha_1}\bar{B}_{i,j}^1\nonumber\\
=&\frac{\alpha_1-\alpha_2}{\alpha_1(1+\alpha_1)^2}(1+\alpha_1)^2e_ie_i^T+\frac{\alpha_2(\alpha_2-\alpha_1)}{(1+\alpha_1)^2}(1+\alpha_1)^2e_je_j^T+\frac{\alpha_2}{\alpha_1}\bar{B}_{i,j}^1\nonumber\\
& ({\rm because\ }\alpha_1\neq-1)\nonumber\\
=&\frac{\alpha_1-\alpha_2}{\alpha_1(1+\alpha_1)^2}\bar{B}_{i,i}^1+\frac{\alpha_2(\alpha_2-\alpha_1)}{(1+\alpha_1)^2}\bar{B}_{j,j}^1+\frac{\alpha_2}{\alpha_1}\bar{B}_{i,j}^1 \ ({\rm by\ }(\ref{eq:Bii})).\label{eq:bij11}
\end{align}

Since $\alpha_1 \not\in \{0, -1\}$, Proposition \ref{Independency} ensures that $\mathcal{\bar{B}}(\alpha_1)$ is linearly independent, and hence, the expression (\ref{eq:bij11}) for ${\bar{B}}_{i,j}^2$ is unique.

Suppose that $\bar{B}_{i,j}^2\in{\rm cone}\left(\bar{\cal B}(\alpha_1)\right)$. In this case, all the coefficients in (\ref{eq:bij11}) should be nonnegative, which implies that 
\begin{align}
\frac{\alpha_1-\alpha_2}{\alpha_1(1+\alpha_1)^2}\ge0,\ \frac{\alpha_2(\alpha_2-\alpha_1)}{(1+\alpha_1)^2}\ge0,\ \frac{\alpha_2}{\alpha_1}>0.\label{eq:con1}
\end{align}

From the last inequality in (\ref{eq:con1}), we have either 
\begin{align*}
{\rm (i)}\ \alpha_1,\alpha_2>0 \hspace{5mm} {\rm or}\hspace{5mm} {\rm (ii)}\ \alpha_1,\alpha_2<0.
\end{align*}
For case (i), from the first and second inequalities of (\ref{eq:con1}), we have $\alpha_2-\alpha_1\ge0$ and $\alpha_1-\alpha_2\ge0$, which implies $\alpha_2= \alpha_1$ and contradicts the assumption $\alpha_2 \neq\alpha_1$. A similar contradiction is obtained for case (ii). Thus, we have $\bar{B}_{i,j}^2\notin{\rm cone}(\bar{\cal B}(\alpha_1))$.\ $\square$
\end{proof}

\subsection{Expression of ${\cal SDD}_n$ with expanded SD bases}
\label{sec:8}

As we have seen in Corollary \ref{convexSDDn}, the set ${\cal SDD}_n={\cal FW}(2)$ is a convex cone. 
This fact ensures that as a corollary of Theorem \ref{unk}, the conical hull of the union of the extended SD bases $\bar{\cal B}(\alpha)$ on $\alpha \in\mathbb{R}$ coincides with ${\cal FW}(2)$ and hence, the set of scaled diagonally dominant matrices ${\cal SDD}_n$:

\begin{corollary}\label{pro:sdd}
\begin{align*}
{\rm cone}\left(\displaystyle\bigcup_{{\alpha}\in\mathbb{R}}\bar{\cal B}(\alpha)\right)={\cal SDD}_n.
\end{align*}
\end{corollary}

\subsection{Notes on the parameter $\alpha$}
\label{sec:alpha}
Here, we discuss the choice for the parameter $\alpha$ to increase the ``volume'' of the polyhedral approximation ${\rm cone}(\bar{\cal B}(\alpha))$ of the semidefinite cone ${\cal S}^n_+$. For any $\alpha\in\mathbb{R}$ and $1\leq i< j\leq n$, by Definition \ref{df:exp}, we can calculate the Frobenius norm of $\bar{B}_{i,j}(\alpha)$:
\begin{align}
\|\bar{B}_{i,j}(\alpha)\|=&\|( e_i+\alpha e_j)( e_i+\alpha e_j)^T\|\nonumber\\
=&\sqrt{{\rm Trace}\left(( e_i+\alpha e_j)( e_i+\alpha e_j)^T( e_i+\alpha e_j)( e_i+\alpha e_j)^T\right)}\nonumber\\
=&\| e_i+\alpha e_j\|^2\nonumber\\
=&1+\alpha^2.\label{bijnorm}
\end{align}

According to Proposition \ref{pro:3.3.3}, by changing $\alpha$, one can obtain different polyhedral approximations. However, we can see that
\begin{align*}
\lim_{|\alpha|\rightarrow\infty}\frac{\bar{B}_{i,j}(\alpha)}{\|\bar{B}_{i,j}(\alpha)\|}&=\lim_{|\alpha|\rightarrow\infty}\frac{1}{1+\alpha^2}(e_i+\alpha e_j)(e_i+\alpha e_j)^T\ (\text{by (\ref{bijnorm})}),\\
&=\lim_{|\alpha|\rightarrow\infty}\left[\frac{1}{1+\alpha^2}e_ie_i^T+\frac{\alpha}{1+\alpha^2}(e_ie_j^T+e_je_i^T)+\frac{\alpha^2}{1+\alpha^2}e_je_j^T\right]\\
&=e_je_j^T= \frac{1}{4}{B}^+_{j,j},
\end{align*}
and by Definitions \ref{df:sdbasis} and \ref{df:exp}, we have 
\begin{align*}
\bar{B}_{i,j}(0)=\frac{1}{4}{B}^+_{i,i},\ \bar{ B}_{i,j}(1)={B}^+_{i,j},\ \bar{B}_{i,j}(-1)={B}^-_{i,j}.
\end{align*}

This shows that, if $|\alpha|\rightarrow\infty$ or $\alpha\in\{0,1,-1\}$, the new matrix $\bar{B}_{i,j}(\alpha)$ will become close to the existing matrices, e.g. ${B}^+_{i,i}$, ${ B}^+_{j,j}$, ${ B}^+_{i,j}$ and ${ B}^-_{i,j}$, and the ``volume'' of the polyhedral approximation ${\rm cone}(\bar{\cal B}(\alpha)\cup{\cal B}_+\cup{\cal B}_-)$ of the semidefinite cone ${\cal S}^n_+$ will also be close to the ``volume'' of the existing inner approximation ${\rm cone}({\cal B}_+\cup{\cal B}_-)$ of ${\cal S}^n_+$.

To give an illustrative explanation of  the above discussion, here we consider the specific case  
\[
{\cal S}^2_+ =\left\{ \left( \begin{array}{ccc}  a & c  \\  c & b  \end{array}  \right) \mid a,b,c\in\mathbb{R}, a,b\ge0,\ ab-c^2\ge0 \right\}
\]
and draw some figures in $\mathbb{R}^3$ with coordinate $a, b$ and $c$. Fig. \ref{fig:angle} [a] shows the set of ${\cal S}^2_+$ in $\mathbb{R}^3$. The red arrow in Fig. \ref{fig:angle} [b] shows the extreme rays $\{\gamma\bar{B}_{i,j}(\alpha)\mid\gamma\ge0\}$ with $|\alpha|\rightarrow\infty$ and $\alpha\in\{0,1,-1\}$. The conical hull of these extreme rays is ${\rm cone}({\cal B}_+\cup{\cal B}_-)$ and its cross section with $\{X\in\mathbb{S}^2\mid\langle X,I\rangle=1\}$ is illustrated as the blue area. To avoid generating a new matrix $\bar{B}_{i,j}(\alpha)$ that is close to the existing matrices, we should choose an $\alpha$ such that the angle between $\bar{B}_{i,j}(\alpha)$ and existing matrices are equal, as illustrated in Fig. \ref{fig:angle} [c]. 

\begin{figure}[htbp]
  \begin{center}
    \begin{tabular}{c}

      % 1
      \begin{minipage}{0.33\hsize}
        \begin{center}
          \includegraphics[clip, width=4cm]{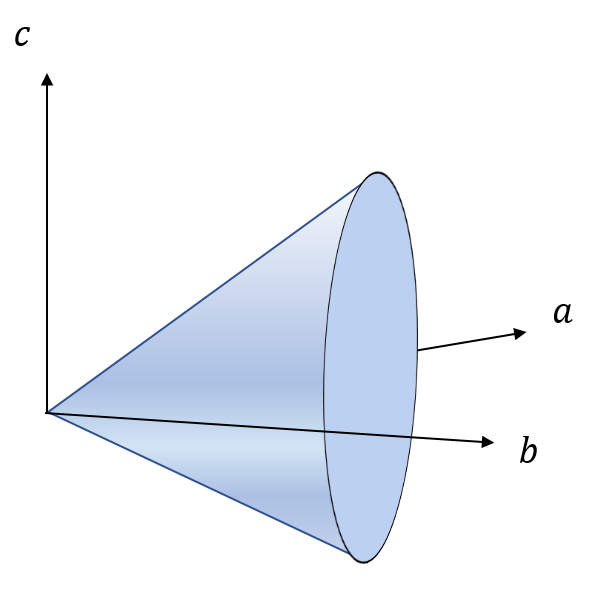}
          \hspace{1.6cm} [a] The set of ${\cal S}^2_+$.
        \end{center}
      \end{minipage}

      % 2
      \begin{minipage}{0.33\hsize}
        \begin{center}
          \includegraphics[clip, width=4cm]{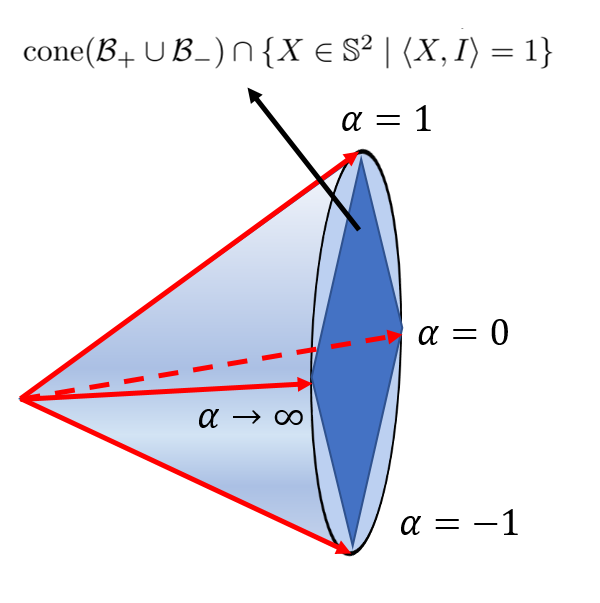}
          \hspace{1.6cm} [b] The set of $\{\gamma\bar{B}_{i,j}(\alpha)\mid\gamma\ge0\}$
        \end{center}
      \end{minipage}
     
     % 3
      \begin{minipage}{0.33\hsize}
        \begin{center}
          \includegraphics[clip, width=4cm]{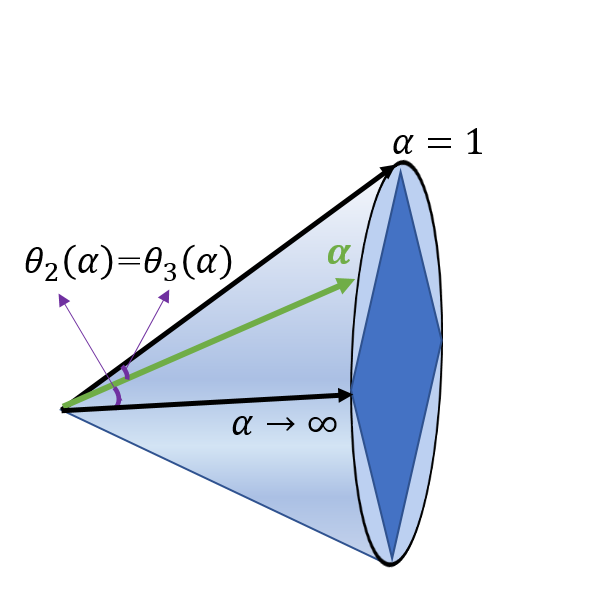}
          \hspace{1.6cm} [c] $\bar{B}_{i,j}(\alpha)$ that are not close to the existing matrices
        \end{center}
      \end{minipage}

    \end{tabular}
    \caption{Choice of $\alpha$ to generate $\bar{B}_{i,j}(\alpha)\in\mathbb{S}^2$ in $\mathbb{R}^3$}
    \label{fig:angle}
  \end{center}
\end{figure}

We expand this idea to the case of generating a matrix $\bar{B}_{i,j}(\alpha)\in\mathbb{S}^n$. 
Given an $\alpha\in\mathbb{R}$, we can define the angles between matrices in the expanded SD bases and SD bases Type I and II for every $1\leq i < j\leq n$, as follows:
\begin{align*}
\theta_1(\alpha):={\rm arccos}\frac{\langle\bar{B}_{i,j}(\alpha) , {B}^+_{i,i}\rangle}{\|\bar{B}_{i,j}(\alpha)\|\|{B}^+_{i,i}\|},\ \theta_2(\alpha):={\rm arccos}\frac{\langle\bar{B}_{i,j}(\alpha) , {B}^+_{j,j}\rangle}{\|\bar{B}_{i,j}(\alpha)\|\|{B}^+_{j,j}\|},\\
\theta_3(\alpha):={\rm arccos}\frac{\langle\bar{B}_{i,j}(\alpha) , {B}^+_{i,j}\rangle}{\|\bar{B}_{i,j}(\alpha)\|\|{B}^+_{i,j}\|},\ \theta_4(\alpha):={\rm arccos}\frac{\langle\bar{B}_{i,j}(\alpha) ,{B}^-_{i,j}\rangle}{\|\bar{B}_{i,j}(\alpha)\|\|{B}^-_{i,j}\|}.
\end{align*}

Thus, we have
\begin{align*}
{\rm cos}\theta_1(\alpha)=&\frac{\langle\bar{B}_{i,j}(\alpha) , {B}^+_{i,i}\rangle}{\|\bar{B}_{i,j}(\alpha)\|\|{B}^+_{i,i}\|}\\
=&\frac{\langle(e_i+\alpha e_j)(e_i+\alpha e_j)^T , (e_i+e_i)(e_i+e_i)^T\rangle}{(1+\alpha^2)\|(e_i+e_i)(e_i+e_i)^T\|}\ (\text{by }(\ref{bijnorm}) )\\
=&\frac{4\|e_i\|^4}{(1+\alpha^2)4\|e_i\|^2}\ (\text{because }e_i^Te_j=0)\\
=&\frac{1}{1+\alpha^2}\ (\text{because }\|e_i\|=1).
\end{align*}
Similarly, we have
\begin{align*}
{\rm cos}\theta_2(\alpha)=\frac{\alpha^2}{1+\alpha^2},\ {\rm cos}\theta_3(\alpha)=\frac{(1+\alpha)^2}{2(1+\alpha^2)},\ {\rm cos}\theta_4(\alpha)=\frac{(1-\alpha)^2}{2(1+\alpha^2)}.
\end{align*}

In general, to obtain a large enough inner approximation with limited parameters, we prefer an $\alpha$ that makes $\theta_1(\alpha)=\theta_3(\alpha)$, which means that the new matrix $\bar{B}_{i,j}(\alpha)$ will be in the middle of ${B}^+_{i,i}$ and ${B}^+_{i,j}$ on the boundary of ${\cal S}^n_+$. Similarly, we can obtain $\alpha$ by calculating $\theta_2(\alpha)=\theta_3(\alpha)$, $\theta_1(\alpha)=\theta_4(\alpha)$ and $\theta_2(\alpha)=\theta_4(\alpha)$. By solving these equalities, we find that 
\begin{align*}
\alpha=\pm1\pm\sqrt2.
\end{align*}
The expansions with these parameters are expected to provide generally large inner approximations for ${\cal S}^n_+$.

\section{Cutting plane methods for the maximum stable set problem}
\label{sec:9}
Conic optimization problems, including SDPs and copositive programs, have been shown to provide tight bounds for NP-hard combinatorial and noconvex optimization problems. Here, we consider applying approximations of ${\cal S}^n_+$ to one of those NP-hard problems, the maximum stable set problem. A stable set of a graph $G(V,E)$ is a set of vertices in $V$, such that there is no edge connecting any pair of vertices in the set. The maximum stable set problem aims to find the stability number, i.e. the number of vertices of the largest stable set of $G$, namely $\alpha(G)$. 

De Klerk and Pasechnik \cite{de2002apprimation} proposed a copositive programming formulation to obtain the exact stability number of a graph $G$ with $n$ vertices:
\begin{align}\label{MSSP}
\alpha(G)=\max&\ \langle ee^T,X\rangle\nonumber\\
{\rm s.t.}&\ \langle A+I,X\rangle=1,\\
&X\in{\cal C}^*_n,\nonumber
\end{align}
where $e$ is the all-ones vector, $A$ is the adjacency matrix of graph $G$, and ${\cal C}^*_n$ is the dual cone of the copositive cone ${\cal C}_n:=\{X\in\mathbb{S}^n\mid d^TXd\ge0 \ \forall d\in\mathbb{R}^n,\ d\ge0\}$.

Although problem (\ref{MSSP}) is a conic optimization problem, it is still difficult since determining whether $X \in {\cal C}^*_n$ or not is NP-hard \cite{dickinson2014computational}. A natural approach is to relax this problem to a  more tractable optimization problem. From the definition of each cone, we can see the validity of the following inclusions:
\begin{eqnarray*}
{\cal C}^*_n\subseteq{\cal S}^n_+\cap{\cal N}^n\subseteq{\cal S}^n_+\subseteq{\cal S}^n_++{\cal N}^n\subseteq{\cal C}_n.
\end{eqnarray*}
By replacing ${\cal C}^*_n$ with ${\cal S}^n_+\cap{\cal N}^n$, one can obtain an SDP relaxation of (\ref{MSSP}):
\begin{align}\label{MSSP1}
\max&\ \langle ee^T,X\rangle\nonumber\\
{\rm s.t.}&\ \langle A+I,X\rangle=1,\\
&\ X\in{\cal S}^n_+\cap{\cal N}^n.\nonumber
\end{align}

Solving this SDP is not as easy as it seems to be; in fact, we could not obtain a useful result of (\ref{MSSP1}) after 6 hours of calculation using the state-of-the-art SDP solver Mosek for a random generalized problem when $n=300$. Combining the expanded SD bases with the cutting plane method, we apply the approximations of ${\cal S}^n_+$ to (\ref{MSSP1}) and solve it by calculating a series of more tractable problems.

Let ${\cal P}^n$ satisfy ${\cal S}^n_+\subseteq{\cal P}^n\subseteq\mathbb{S}^n$ and replace $X\in {\cal S}^n_+$ by $X\in{\cal P}^n$ in (\ref{MSSP1}). Then, we obtain a relaxation of (\ref{MSSP1}):
\begin{align}\label{MSSP2}
\max&\ \langle ee^T,X\rangle\nonumber\\
{\rm s.t.}&\ \langle A+I,X\rangle=1,\\
&\ X\in{\cal P}^n\cap{\cal N}^n.\nonumber
\end{align}

Usually, the relaxed problem (\ref{MSSP2}) is expected to be easier to solve and to give us a better upper bound of problem (\ref{MSSP1}) from its optimal solution $X^*$. 
To get a better upper bound, we select some eigenvectors with negative eigenvalues of an optimal solution $X^*$ of problem  (\ref{MSSP2}), say $d_1,..,d_k$, by adding cutting planes
\begin{align*}
\langle d_id_i^T,X\rangle\ge0\ \ \  ( i=1,..,k)
\end{align*}
to (\ref{MSSP2}), and obtain a new optimization problem
 \begin{align}\label{MSSP3}
\max&\ \langle ee^T,X\rangle\nonumber\\
{\rm s.t.}&\ \langle A+I,X\rangle=1,\\
&\ \langle d_id_i^T,X\rangle\ge0\  ( i=1,..,k)\nonumber\\
&\ X\in{\cal P}^n\cap{\cal N}^n.\nonumber
\end{align}

Notice that the optimal solution $X^*$ of problem (\ref{MSSP2}) is cut from the feasible region of problem   (\ref{MSSP3}) since $\langle d_id_i^T,X^*\rangle<0\  ( i=1,..,k)$. On the other hand, since ${\cal S}^n_+=\{X\in\mathbb{S}^n\mid \forall d\in\mathbb{R}^n, \ \langle dd^T,X\rangle\ge0 \} \subseteq {\cal P}^n$, every feasible solution of (\ref{MSSP1}) is feasible for (\ref{MSSP3}), and hence problem  (\ref{MSSP3}) is a relaxation of problem (\ref{MSSP1}). These facts ensure that problem (\ref{MSSP3}) is a tighter relaxation of problem (\ref{MSSP1}) than problem  (\ref{MSSP2}). By repeating this procedure, we are able to obtain a series of nonincreasing upper bounds of (\ref{MSSP1}). Since the eigenvectors are usually dense, we only have to add eigenvectors corresponding to up to the second smallest eigenvalues to $\{d_i\}$ at every iteration, which increases computational efficiency. 

As for the selection of the initial relaxation ${\cal P}^n$, we are ready to use the approximations of ${\cal S}^n_+$ based on the expanded SD bases. Let ${\cal H}:=\{\pm1,\pm1\pm\sqrt2\}$ be the set of parameters calculated in Section \ref{sec:alpha}, and let ${\cal SDB}_n$ denote the conical hull of expanded SD bases using ${\cal H}$:
\begin{align*}
{\cal SDB}_n:={\rm cone}\left(\displaystyle\bigcup_{{\alpha}\in{\cal H}}\bar{\cal B}(\alpha)\right).
\end{align*}
Then, as has been described in the previous sections, we have
\begin{eqnarray}\label{inclusion}
{\cal S}^n_+\subseteq{\cal SDD}^*_n\subseteq{\cal SDB}^*_n\subseteq{\cal DD}^*_n.
\end{eqnarray}

If ${\cal SDB}^*_n$ or ${\cal DD}^*_n$ is selected to be ${\cal P}_n$, the corresponding relaxed problem in the cutting plane procedure becomes an LP, which allows us to use powerful state-of-the-art LP solvers, such as Gurobi \cite{gurobi}. Ahmadi et. al. \cite{ahmadi2017optimization} showed that when ${\cal SDD}^*_n$ is selected, the relaxations turn out to be SOCPs. Although ${\cal SDD}^*_n$ provides a tighter relaxation than either ${\cal DD}_n$ or  ${\cal SDB}_n$, the latter two relaxations are expected to have a lower computational cost. In addition, in \cite{ahmadi2017optimization}, Ahmadi et al. also proposed an SOCP-based cutting plane approach, named SDSOS, which adds SOCP cuts at every iteration. We conducted experiments to compare the efficiencies of those cutting plane methods using different approximations and SDSOS. The specifications of the experimental methods are summarized in Table \ref{tab:methods description}.

\begin{table}[htb]
\label{tab:methods description}
\begin{center}
\caption{Specifications of the experimental methods}
\begin{tabular}{ccccc}
\hline
\multirow{2}{*}{Method} & \multicolumn{1}{c}{\multirow{2}{*}{${\cal P}^n$}} & \multicolumn{2}{c}{Number of cuts added at each iteration} & \multirow{2}{*}{Solver} \\
                              & \multicolumn{1}{c}{}                                              & LP cut              & SOCP cut &                         \\\hline
CPDD                      & ${\cal DD}^*_n$                                                     & 2                      & 0  & Gurobi                  \\
CPSDB                    & ${\cal SDB}^*_n$                                                & 2                       & 0  & Gurobi                  \\
CPSDD                    & ${\cal SDD}^*_n$                                                   & 2                       & 0 & Mosek                   \\
SDSOS                    & ${\cal SDD}^*_n$                                                   & 2                       & 1  & Mosek                  \\\hline
\end{tabular}
\end{center}
\end{table}

We tested these methods on the Erd$\ddot{\rm o}$s-R\'{e}nyi graphs $ER(n,p)$, randomly generated by Ahmadi et al. in \cite{ahmadi2017optimization}, where $n$ is the number of vertices and every pair of vertices has an edge with probability $p$. All experiments were performed with MATLAB 2018b on a Windows PC with an Intel(R) Core(TM) i7-6700 CPU running at 3.4 GHz and 16 GB of RAM. The LPs were solved using Gurobi Optimizer 8.0.0 \cite{gurobi} and the SOCPs and SDPs are solved using Mosek Optimizer 9.0 \cite{mosek}. 

Fig. \ref{fig:250_8_ite} shows the result for an instance with $n=250$ and $p=0.8$. The x-axis is the number of iterations, and the y-axis is the gap between the upper bounds of each method and the SDP bound obtained by (\ref{MSSP1}); the gap is computed by $\left|\frac{f^*-f_k}{f^*}\right|\times100\%$ for the obtained upper bound  $f_k$ at $k$'s iteration and the SDP bound $f^*$ obtained by solving (\ref{MSSP1}) directly.

As can be seen in this figure, the accuracy of CPDD is the worst among the four methods at each iteration. CPSDB achieves almost the same upper bounds as CPSDD and SDSOS, which shows that the proposed polyhedral approximation ${\cal SDB}_n$ is promising for obtaining a solution close to the non-polyhedral approximation ${\cal SDD}_n$ of ${\cal S}^n_+$. Although SDSOS adds an extra SOCP cut at every iteration and takes longer  to solve, the accuracy of SDSOS does not seem to be affected and is not so different from the accuracy of CPSDD at each iteration.

\begin{figure}[htbp]
\begin{center}
\includegraphics[width=11cm,clip]{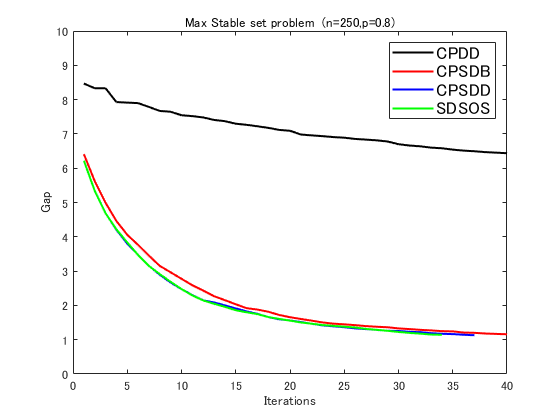}
\begin{minipage}{0.71\textwidth}
\caption{Relation between the number of iterations and the gap}
\label{fig:250_8_ite}
\end{minipage}
\end{center}
\end{figure}

Fig. \ref{fig:250_8} shows the relation between the computation time and the gap of each method for the same instance. Although its accuracy is not necessarily the best at every iteration, it seems that CPSDB is the most efficient method. CPSDB attains an upper bound whose gap is $2$ within $30$ s, while CPSDD and SDSOS  attain upper bounds whose gap is $4$ after the same amount of time. The difference might come from that the subproblems of CPSDB are sparse LPs at earlier iterations and the computations are relatively cheaper than those of CPSDD and SDSOS whose subproblems are SOCPs.

\begin{figure}[htbp]
\begin{center}
\includegraphics[width=11cm,clip]{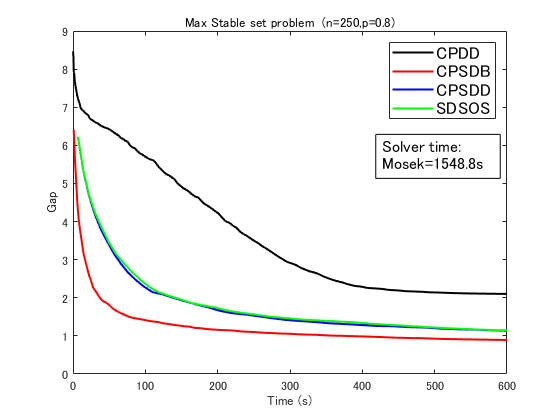}
\begin{minipage}{0.71\textwidth}
\caption{Relation between the computational time (s) and the gap}
\label{fig:250_8}
\end{minipage}
\end{center}
\end{figure}

Table \ref{Table:SDP} and \ref{Table:LP} give the bounds of iterative methods and the SDP bound for all the instances. In Table \ref{Table:SDP}, the CPSDD$_0$/SDSOS$_0$ column shows the first upper bound obtained by CPSDD and SDSOS, i.e., the upper bound obtained by solving the same SOCP before adding any cutting plane.
The (5 min) and (10 min) columns of CPSDD (SDSOS) show the upper bounds obtained after $5$ minutes and after $10$ minutes of the CPSDD (SDSOS) computation, respectively. The SDP column shows the SDP bound obtained by solving (\ref{MSSP1}). 

Similarly, in Table \ref{Table:LP}, the CPDD$_0$ and CPSDB$_0$ columns show the first upper bounds obtained by  CPDD and CPSDB, respectively, before adding any cutting plane.
The (5 min) and (10 min) columns of CPDD (CPSDB) show the upper bounds obtained after $5$ minutes and after $10$ minutes of the CPDD (CPSDB) computation, respectively. 

Note that we failed to solve SDPs (\ref{MSSP1}) for instances having $n=300$ nodes within our time limit $20000s$. In Table \ref{Table:SDP},  the Value and Time (s) columns of SDP with $n=300$ show the results obtained in \cite{ahmadi2017optimization} for these two instances, as a reference.

As can be seen in Table \ref{Table:SDP} and \ref{Table:LP}, for all instances, the values of CPSDD$_0$/SDSOS$_0$ are better than the values of CPSDB$_0$ and CPDD$_0$. These results correspond to the inclusion relationship of initial approximations (\ref{inclusion}). We can also see that the values of CPSDB$_0$ are almost the same as those of CPSDD$_0$/SDSOS$_0$ for all instances, while the values of CPDD$_0$ are much worse than others. For all instances, CPSDB seems to be significantly more efficient than all other methods. For example, for instance with $n=250$ and $p=0.3$, after $10$ min of calculation, CPSDB obtained an upper bound of $73.24$, while CPSDD and SDSOS got upper bounds greater than $90$ and CPDD got a bound of more than $146$. 

At present, solving a large SDP, e.g., one with more than $n=300$ nodes requires a significant amount of computational time. The cutting plane method CPSDB with our polyhedral approximation ${\cal SDB}_n$ is a promising way of obtaining efficient upper bounds of such large SDPs in a moderate time.

\begin{table}[]
\caption{Upper bounds obtained by SDP and SOCP methods on $ER(n,p)$ graphs}
\begin{tabular}{rrrrrrrrrr}
\hline
\multirow{2}{*}{n} & \multirow{2}{*}{p} & \multicolumn{2}{c}{CPSDD$_0$/SDSOS$_0$} & \multicolumn{2}{c}{CPSDD} & \multicolumn{2}{c}{SDSOS} & \multicolumn{2}{c}{SDP} \\
                   &                    & Value              & Time (s)            & (5 min)      & (10 min)     & (5 min)      & (10 min)     & Value     & Time (s)     \\ \hline
150                & 0.3                & 105.70             & 0.95               & 38.91       & 37.02       & 40.97       & 37.38       & 20.44     & 105.46      \\
150                & 0.8                & 31.78              & 1.00               & 10.07       & 9.66        & 9.70        & 9.31        & 6.00      & 110.63      \\
200                & 0.3                & 140.47             & 3.14               & 70.48       & 55.52       & 75.46       & 61.31       & 23.73     & 549.63      \\
200                & 0.8                & 40.92              & 3.14               & 12.10       & 11.29       & 12.17       & 11.38       & 6.45      & 497.55      \\
250                & 0.3                & 176.25             & 6.60               & 115.41      & 93.81       & 119.67      & 99.99       & 26.78     & 1562.52     \\
250                & 0.8                & 51.87              & 6.79               & 17.36       & 15.30       & 17.43       & 15.39       & 7.18      & 1553.63     \\
300                & 0.3                & 210.32             & 13.05              & 160.42      & 138.60      & 162.77      & 143.12      & (29.13)         & (32300.60)           \\
300                & 0.8                & 60.97              & 13.31              & 21.71       & 17.77       & 22.66       & 18.50       & (7.65)         & (20586.02)           \\ \hline
\end{tabular}
\label{Table:SDP}
\end{table}

\begin{table}[]
\caption{Upper bounds obtained by LP methods on the same $ER(n,p)$ graphs}
\begin{tabular}{rrrrrrrrrr}
\hline
\multirow{2}{*}{n} & \multirow{2}{*}{p} & \multicolumn{2}{c}{CPDD$_0$} & \multicolumn{2}{c}{CPDD} & \multicolumn{2}{c}{CPSDB$_0$} & \multicolumn{2}{c}{CPSDB} \\
                   &                    & Value        & Time (s)       & (5 min)     & (10 min)     & Value         & Time  (s)       & (5 min)      & (10 min)     \\ \hline
150                & 0.3                & 117          & 0.06          & 76.76      & 67.51       & 107.29        & 0.24          & 36.80       & 35.12       \\
150                & 0.8                & 46           & 0.05          & 13.70      & 12.71       & 32.76         & 0.28          & 9.51        & 9.06        \\
200                & 0.3                & 157          & 0.1           & 113.28     & 104.07      & 142.25        & 0.52          & 55.07       & 48.18       \\
200                & 0.8                & 54           & 0.11          & 17.39      & 16.07       & 42.14         & 0.57          & 11.58       & 11.00       \\
250                & 0.3                & 194          & 0.17          & 154.75     & 146.20      & 178.30        & 0.84          & 91.88       & 73.24       \\
250                & 0.8                & 68           & 0.17          & 28.02      & 22.26       & 53.22         & 1.00          & 14.76       & 13.57       \\
300                & 0.3                & 230          & 0.26          & 183.89     & 174.02      & 212.97        & 1.29          & 133.83      & 110.95      \\
300                & 0.8                & 78           & 0.24          & 47.87      & 32.28       & 62.47         & 1.36          & 18.11       & 16.05       \\ \hline
\end{tabular}
\label{Table:LP}
\end{table}

\section{Concluding remarks}
\label{sec:10}

We developed techniques to construct a series of sparse polyhedral approximations of the semidefinite cone. We provided a way to approximate the semidefinite cone by using SD bases and proved that the set of diagonally dominant matrices can be expressed with sparse SD bases. We proposed a simple expansion of SD bases that keeps the sparsity of the matrices that compose it. We gave the conditions for generating linearly independent matrices in expanded SD bases as well as for generating an expansion different from the existing one. We showed that the polyhedral approximation using our expanded SD bases contains the set of diagonally dominant matrices and is contained in the set of scaled diagonally dominant matrices. We also proved that the set of scaled diagonally dominant matrices can be expressed using an infinite number of expanded SD bases. 

The polyhedral approximations were applied to the cutting plane method for solving a semidefinite relaxation of the maximum stable set problem. The results of the numerical experiments showed that the method with our expanded SD bases is more efficient than other methods  (see Fig. \ref{fig:250_8}); improving  the efficiency of our method still remains an important study issue.

One future direction of study is to increase the number of vectors in the definition of the SD bases. The current SD bases are defined as a set of matrices $(e_i+e_j)(e_i+e_j)^T$. If we use three vectors, as in $(e_i+e_j+e_k)(e_i+e_j+e_k)^T$, we might obtain another inner approximation that remains relatively sparse when the dimension $n$ is large.

Another future direction is to focus on the factor width $k$ of a matrix. The cone of matrices with factor width at most $k=2$ was introduced in order to give another expression of the set $\mathcal{SDD}_n$ of scaled diagonally dominant matrices. By considering a larger width $k > 2$, we may obtain a larger inner approximation of the semidefinite cone $\mathcal{S}^n_+$, although it would not be polyhedral, or even characterized by using SOCP constraints. Finding efficient ways to solve approximation problems over such cones might be an interesting challenge.

Also, our expanded SD bases can be applied to some other difficult problems. Mixed integer nonlinear programming has recently become popular in many practical applications. In \cite{lubin2018polyhedral}, Lubin et al. proposed a cutting plane framework for mixed integer convex optimization problems. In \cite{kobayashibranch}, Kobayashi and Takano proposed a branch and bound cutting plane method for mixed integer SDPs. It would be interesting to see whether the approximations of ${\cal S}^n_+$ proposed in this paper could be used to improve the efficiency of those methods.

\section*{Acknowledgments}
 The authors would like to sincerely thank the anonymous reviewers for their thoughtful and valuable comments which have significantly improved the paper. Among others, one of the reviewers pointed out Remark \ref{PtoI} which helped the authors to simplify the presentation of the paper. This research was supported by the Japan Society for the Promotion of Science through a Grant-in-Aid for Challenging Exploratory Research (17K18946) and a Grant-in-Aid for Scientific Research ((B)19H02373) from the Ministry of Education, Culture, Sports, Science and Technology of Japan.

%\bibliographystyle{siamplain}
%\bibliography{ref}

\end{document}